\documentclass[11pt]{amsart}
\usepackage[T1]{fontenc}

\usepackage{lmodern, amsfonts,amsmath,amstext,amsbsy,amssymb,
amsopn,amsthm,upref,eucal,mathptmx}

\usepackage{mathrsfs}

\RequirePackage{xcolor} 
\definecolor{halfgray}
{gray}{0.55}
\definecolor{webgreen}
{rgb}{0,0.4,0}
\definecolor{webbrown}
{rgb}{.8,0.1,0.1}
\definecolor{red}
{rgb}{1,0,0}
\usepackage{microtype}

\newcommand \R {{ \mathbb R}}
\def\C{{\mathbb C}}

\newcommand \N {{ \mathbb N}}

\newcommand \re {{%
\operatorname{Re}
}}
\newcommand \im {{%
\operatorname{Im}
}}

\newtheorem{theorem}{Theorem}[section]
\newtheorem {lemma} [theorem]{Lemma}
\newtheorem {proposition}[theorem]{Proposition}
\newtheorem{corollary}[theorem]{Corollary}

\newtheorem{question}[theorem]{Question}

\title[Finite codimension stability of invariant surfaces]%
{Finite codimension stability of invariant surfaces}

  \author{Giovanni Forni}

\address{Department  of Mathematics\\
  University of Maryland \\
  College Park, MD USA}
  
\email
    {gforni@math.umd.edu}

\date{\today}

\begin{document}

\def\echo#1{\relax}
    
\begin{abstract}
Following recent work of T.~Alazard and C.~Shao \cite{AlSh} on applications of para-differential calculus  to smooth conjugacy and stability problems for Hamiltonian systems, we  prove finite codimension stability  of invariant surfaces (in finite differentiability classes) of flat geodesic flows on translation surfaces.  
The result is also based  on work of the author \cite{F97},
\cite{F21} on the cohomological equation for translation flows.
\begin{sloppypar}
  \end{sloppypar}
\end{abstract}
\maketitle
\section{Introduction}

The billiard in a completely integrable rational polygon, as well the geodesic flow on a flat torus, are basic examples of integrable Hamiltonian
systems: their phase space is entirely foliated by invariant $2$-dimensional tori on which the flow is linear. In this context the KAM theory 
implies that, under sufficiently small smooth perturbations of the Hamiltonian, a positive measure set of invariant tori persists and therefore the
perturbed system is not ergodic. 

\smallskip
In this paper, we prove an analogous result for a class of pseudo-integrable systems: (non-integrable) billiards in rational polygons and, more generally,
geodesic flows for the flat metric on translation surfaces. The phase space of these systems is foliated by invariant surfaces of higher genus (in the
non-integrable case) and it is natural to ask whether any of such surfaces persists under small smooth perturbations of the Hamiltonian.

\smallskip
With this problem in mind the author proved a result \cite{F97} on the linearized problem, that is, on the so-called cohomological equation for translation flows
on higher genus surfaces. This work found that, contrary to the case of (Diophantine) linear toral flows, the Lie derivative operator for translation flows
on higher genus surface has range of finite codimension in every finite differentiability class, and of infinite codimension in the space of infinitely differentiable
functions. 

In addition, the obstructions to the existence of solutions depend on the specific translation flow considered, while in the case of the torus
the Lie derivative operator has range of codimension one, transverse to the space of constant functions, and the only obstruction, the mean, is independent
of the linear flow. 

\smallskip
As a consequence of this result, it was natural to conjecture, transposing the  results on the linearized problem to the non-linear problems, that typical translation flows would
be stable with finite codimension in any finite differentiability class under smooth perturbations, and that the typical invariant surfaces of pseudo-integrable 
systems would be similarly stable with finite codimension, hence there would exists a finite codimension family of non-ergodic perturbations. 

\smallskip
Unfortunately the presence of  distributional obstructions of growing Sobolev order creates serious difficulty in the application of the KAM or Nash-Moser  iteration
method, which we have not been overcome to date. 

\smallskip
These difficulties are caused by the feature of the Nash-Moser iteration of letting the high norms grow (while controlling
the decay of low norms), as a consequence of the application of smoothing operators at each step of the iteration. Therefore the values of the higher order obstructions on the data
of the linearized equation may blow up in the iteration process which therefore may fail to convergence. 

\smallskip
For the conjugacy problem, S.~Marmi, P.~Moussa and J.-C.Yoccoz \cite{MMY12} were able to bypass the difficulty
by a method inspired to M.~Herman's fixed point solution of the conjugacy problem for typical (Roth-type) circle rotations  \cite{He85}, based on a ``Schwarzian derivative trick''. 

\smallskip
The work \cite{MMY12}  was made possible by improved regularity estimates for solutions of the cohomological equation for Interval Exchange Transformations (and translation flows) derived in \cite{MMY05} by a dynamical approach completely different from the harmonic analysis methods of \cite{F97}. 

In fact, Herman's method requires a loss of at most $2-$ derivatives for the solutions of the cohomological equation, a result which is out of the reach of the method of \cite{F97}  (see also \cite{F21}).  In \cite{MMY05} the loss of derivatives is (at most) $1+ BV$, while in \cite{F97} is at best $3 +$ (and this result is only achieved in \cite{F21}).

\smallskip
For low regularity conjugacies (that is, for conjugacies of class $C^1$) and for the global conjugacy problem, S. Ghazouani \cite{Gh21}  and S.~Ghazouani and C.~Ulcigrai \cite{GU23} 
have developed a renormalization approach. In \cite{Gh21} a $C^1$ conjugacy result is proved for Interval Exchange Transformations of periodic type, while in \cite{GU23} (followed
by a refined result \cite{GU25} on the regularity of the conjugacy),
the authors prove a rigidity result, that is, a higher genus version of Herman's global linearization theorem for circle diffeomorphisms \cite{He79}.

\smallskip
The analogous question on the stability of invariant (higher genus) surfaces under smooth perturbations has so far remained open. 

\medskip
In this paper we follow the para-differential calculus approach of T.~Alazard and C.~Shao \cite{AlSh} and give a proof of finite codimension stability of a typical invariant
surface. Indeed, the para-differential approach appears especially powerful to treat non-linear problems with finite codimension, increasing with the regularity class
or with infinite codimension, for which a KAM approach is problematic and has not so far been implemented. We believe that the para-differential method
could lead to another proof of the linearization result \cite{MMY12}, possibly with somewhat stronger regularity assumptions, but under ``Diophantine'' conditions weaker than
the Roth-type condition assumed there.

\medskip
Our main result is the following 
\begin{theorem}
\label{thm:main}
Let $(M, \omega)$ any translation surface. For almost all $\xi \in P^1(\R^2)$, the invariant $2$-dimensional surface $M_\xi$ in a given energy level for the flat geodesic 
flow given by all tangent vector parallel to $\xi$ and fixed norm, is stable with finite codimension under smooth perturbations in the the following sense.  

There exists $s_0>0$ such that 
for all $s>s_0$ there exists a local subvariety  $\mathcal H_s(\xi)$ (a priori dependent on $\xi \in P^1(\R^2)$)  of finite codimension $h_s\in \N$ of the space of Hamiltonians  such that for all Hamiltonians 
sufficiently close to the Hamiltonian$H_0$ of the flat geodesic flow in the Sobolev space $H^s(M)$ and equal to $H_0$  near the set $\Sigma:=\{\omega =0\}$, 
the Hamiltonian flow of $H$ has an invariant surface $M^H_\xi$ of genus equal to he genus of $M$. The invariant surface   $M^H_\xi$ is an $H^t(M)$-graph over 
$M_\xi$  for $t<s-s_0$ and the Hamiltonian flow of $H$ on $M^H_\xi$ is $H^t(M)$-conjugated to the translation flow given by the restriction of the flat geodesic flow to $M_\xi$. 
The codimension $h_s$ of the subvariety $\mathcal H_s(\xi)$ grows linearly in $s>s_0$ and in the genus of the surface $M$ and the cardinality of $\Sigma$.
\end{theorem} 

From Theorem \ref{thm:main} we derive the following result for billards in rational polygons.

\begin{corollary} For any  rational polygon $P$ and for almost all $\xi \in P^1(\R^2)$, the invariant $2$-dimensional surface $M_\xi$ for the billiard flow in $P$, endowed with the flat metric $R_0$,   in direction $\xi$ and in a given energy level, is stable with finite codimension under smooth perturbations in the the following sense.   
There exists $s_0>0$ such that for all $s>s_0$ there exists a local subvariety  $\mathcal K_s(\xi) $ (a priori dependent on $\xi \in P^1(\R^2)$)  of finite codimension $k_s\in \N$ of the space of metric on $P$ sufficiently close to the  flat metric $R_0$
 in the Sobolev space $H^s(M)$, and equal to the flat metric near its corners, such that the billiard flow in $P$ with respect to any metric $R\in \mathcal K_s$  has 
 an invariant surface $M^R_\xi$ of genus equal to the genus of the unfolding of $P$. The invariant surface $M^R_\xi$ is an $H^t(M)$-graph over 
$M_\xi$  for $t<s-s_0$ and the billiard flow for $P$, endowed with the metric $R$,  on $M^R_\xi$ is $H^t(M)$-conjugated to the translation flow given by the restriction 
of the flat metric billiard flow to $M_\xi$. 
The codimension $k_s$ of the subvariety $\mathcal K_s(\xi)$ grows linearly in $s>s_0$ and in the genus  and the cardinality of the set of
conical singularities of the translation surface $M_P$ unfolding of~$P$.

\end{corollary} 

\medskip
\noindent{\bf Acknowledgments.} The author is grateful Carlos Matheus for first telling him of the preprint~\cite{AlSh} and suggesting that the method may apply in the
context Interval Exchange Transformations or translation flows, and to T.~Alazard and C.~Shao for several discussions of their work and its potential applications, and
for suggestions about background and relevant literature. The author is grateful to N.~Tedesco for pointing out a calculation mistake in Lemma \ref{lemma:lin_id}
and to C.~Shao and N. Tedesco for the subsequent discussion about the formulas in the lemma.

\section{Sobolev spaces and para-differential operators}

In this section we recall the definition of the natural (weighted) Sobolev spaces on translation surfaces (see \cite{F97}, \cite{F21}), and  extend the para-differential 
formalism to translation surfaces (following \cite{AlSh}) . 

Let $(M, \omega)$ be a translation surface. Let $L^2(M, \omega)$ denote the space of square-integrable functions with respect to to the area form 
$dA_\omega = -(i/2) \omega \wedge \bar\omega$. Let $X$ and $Y$ denote the horizontal and vertical vector fields defined by the conditions
$$
\imath_X \re(\omega) = - \imath_Y \im(\omega) =  1\,, \quad \text{ and } \quad \imath_X \im(\omega) = \imath_Y \re(\omega) = 0\,.
$$
With respect to a canonical coordinate $z$ centered at a regular point of $(M, \omega)$ (that is, a coordinate such that $\omega= dz$, we 
have $dA_\omega =  dx \wedge dy$ and 
$$
X= \frac{\partial}{\partial x} \quad \text{ and } \quad Y= \frac{\partial}{\partial y}  \,.
$$
With respect to a canonical coordinate $z$ centered at a cone point of $(M, \omega)$ of angle $2\pi (k+1)$ (that is, a coordinate such that $\omega= z^k dz$, we  
have $dA_\omega =\vert z \vert^k dx \wedge dy$ and 
$$
X=  \vert z \vert^{-2k} \Big( \re(z^k)  \frac{\partial}{\partial x}  - \im(z^k) \frac{\partial}{\partial y}\Big)   \quad \text{ and } \quad 
Y=  \vert z \vert^{-2k} \Big( \im(z^k) \frac{\partial}{\partial x}  + \re(z^k) \frac{\partial}{\partial y}\Big)   \,.
$$
In other terms, for any $p\in M$ a cone point of angle $2\pi (k+1)$, the map $\pi_k : U(p) \to D \subset \C$ defined on a neighbourhood
$U(p) \subset M$ such that $U(p) \cap \Sigma=\{p\}$ with respect to a canonical coordinate $z: U(p) \to D \subset \C$ as
$$
\pi_k (z) =  \frac{ z^{k+1} }{k+1} \,, \quad \text{ for all } z\in U(p)\,,
$$
is a $(k+1)$-fold branched cover of a neighborhood $D$ of $0\in \C$ such that $\pi_k^\ast (dz) = \omega \vert U(p)$ and
$$
(\pi_k)_\ast (X) = \frac{\partial}{\partial x} \quad \text{ and } \quad  (\pi_k)_\ast  (Y) = \frac{\partial}{\partial y}  \,.
$$
The (weighted) Sobolev spaces $H^s_\omega (M)$ (for $s \geq 0$)  can be defined as the subspaces of $f\in L^2(M, \omega)$ such that 
$u \in H^s_{loc} (M\setminus \Sigma)$ and for every cone point $p\in \Sigma$ (of angle $2\pi (k+1)$) there exists a function 
$F \in H^s(D)$ such that $f = (\pi_k)^\ast (F)$.  In short, for any translation atlas $\mathcal U:=\{(U, \pi_U)\}$ on $M$ (composed of charts
given by canonical coordinates)
$$
f \in H^s_\omega (M)  \Longleftrightarrow     f \vert U \in (\pi_U)^\ast  ( H^s ( \pi_U (U) ) \quad \text{ for all  } (U, \pi_U) \in \mathcal U\,.
$$
The norm on the space $H^s_\omega(M)$ can be defined for $s \in \N$ as follows:
$$
\Vert f \Vert _{H^s_\omega(M)}^2 = \sum_{\alpha+\beta \leq s}  \Vert X^\alpha Y^\beta  \Vert^2_{L^2(M, \omega)} \,, \quad \text{ for all } f \in H^\infty_\omega(M)\,.
$$
Another possible definition of Sobolev norms on translation surfaces is in terms of fractional powers of the Friederichs Laplacian. Let $\Delta_F$ denote
the Friederichs extension of the flat Laplacian with domain $H^\infty_\omega(M)$ and let $\{\lambda_n\}_{n\in \N}$ denote the sequence of 
eigenvalues of the negative of the Friederichs Laplacian $-\Delta_F$ and let $\{ e_n\}$ a corresponding orthonormal system of eigefunctions. 

The Friederichs (weighted) Sobolev norms (for all $s \in \R$) are 
$$
\Vert f \Vert _{\bar H^s_\omega(M)}^2 =  \sum_{n\in \N} (1+ \lambda_n)^s \vert \langle f, e_n\rangle\vert^2 \,, \quad \text{ for all } f \in H^\infty_\omega(M)\,.
$$
By definition the Friederichs Sobolev norms are interpolation norms. 

\smallskip
The fractional (weighted) Sobolev norms can then be defined as follows: for any $s= k + \sigma \geq 0$ with $k\in \N$ and $\sigma\in [0, 1)$, we define,
for all $ f \in H^\infty_\omega(M)$, 
$$
\Vert f \Vert_{H^s_\omega(M)}^2 = \Vert f \Vert_{H^{k}_\omega(M)}^2 + \sum_{\alpha + \beta=k}  \Vert X^\alpha Y^\beta f \Vert_{\bar H^{\sigma}_\omega(M)}^2
+ \Vert  Y^\alpha X^\beta f \Vert_{\bar H^{\sigma}_\omega(M)}^2  \,.
$$
We have the following comparison result between weighted Sobolev norms, Friederichs weighted Sobolev norms and standard Sobolev norms:
\begin{lemma} (\cite{F21}, Lemma 2.11 )
The following continuous embedding and isomorphisms of Banach spaces hold:
\begin{itemize}
\item  $H^s(M) \subset  H^s_\omega(M)   \equiv  \bar H^s_\omega(M)$, \qquad for $0\leq s <1$;
\item  $H^s(M) \equiv H^s_\omega(M)   \equiv  \bar H^s_\omega(M)$,  \qquad for $s =1$;
\item $H^s_\omega(M)   \subset  \bar H^s_\omega(M) \subset H^s(M)$,  \qquad for $s>1$\,.
\end{itemize} 
For $s\in [0,1]$ the space $H^s(M)$ is dense in $H^s_\omega(M)$ and, for $s > 1$, the closure of $H^s_\omega(M)$  in $\bar H^s_\omega(M)$ or $H^s(M)$ 
has finite codimension.
\end{lemma}
The weighted Sobolev spaces $H^{-s}_\omega(M)$  with negative exponents 
are defined as the dual spaces of the spaces $H^s_\omega(M)$ for all $s>0$.

\medskip
Para-differential operators on euclidean spaces were introduced by M. Bony in \cite{Bo81}  (see also \cite{Me08}).
Para-differential operators can be generalized to smooth compact manifolds working in local coordinates. 

A recent detailed introduction
of para-differential calculus on compact manifolds can be found in \cite{BGdP21}, \S 2, in \cite{Del15}, \S 3, in Chap. 6 of \cite{Sh22} 
(and to some extent earlier in \cite{Tay91},  Chap. 3, \cite{Tay11}, Chap 13. 10). 

Para-differential operators can therefore be extended to translation surfaces by defining para-differential operators locally with respect to canonical
charts. Since all the results are local and weighted Sobolev spaces are defined in terms of canonical charts, they generalize to our context 
first for functions in $H^\infty_\omega(M)$, then by continuity to functions in the spaces $H^s_\omega(M)$.  
In particular we can define para-products 
$Op(a):=T_a$  for functions $a \in L^\infty (M)$. 

\smallskip
We have the following results (see \cite{AlSh}, Props. 2.1 -2.3 and Prop. 3.5):

\begin{proposition} [Continuity of para-product operators]  \label{prop:para-cont}  If  $a\in L^\infty(M)$,  then $T_a$ is a bounded linear
operator from $H^s_\omega(M)$ to itself, and in fact there exists a constant $C_s>0$ such that
$$
\Vert T_a \Vert_{\mathcal L( H^s_\omega(M), H^s_\omega(M)) } \leq C_s \Vert a \Vert_{L^\infty(M)}\,.
$$
\end{proposition}

Let now $C^r_\omega(M)$ denote the space of functions which belong to the Zygmund (or Lipschitz) space $C^r_\ast$ locally with respect to
canonical coordinates on $(M, \omega)$. 

\begin{proposition} [Composition of para-product operators]   \label{prop:para-comp}  If  $a, b\in C^r_\omega (M)$,  then $T_{ab} -T_a T_b$ is a bounded linear
operator from $H^s_\omega(M)$ to $H^{s+r}_\omega(M)$ , and in fact there exists a constant $C_{r,s}>0$ such that
$$
\Vert T_{ab} -T_a T_b \Vert_{\mathcal L( H^s_\omega(M), H^{s+r}_\omega(M)) } \leq C_{r,s}  \Vert a \Vert_{C^r_\omega} \Vert b \Vert_{C^r_\omega}\,.
$$
\end{proposition}

\begin{proposition}[Para-linearization]   \label{prop:para-lin}    Let $s>1$ and   let $N_s\in \N$
denote the smallest integer such that $N_s > 2s-1$. For any functions $u \in  H^s_\omega(M, \R^2)$ and
$F := F(x,u) \in  C^{N_s+3}_\omega(M \times \R^2 )$, the following para-linearization formula holds:
$$
F(x,u)-F(x,0)= Op({ \frac{\partial F(x,u)}{ \partial u}} ) u + \mathcal R_{PL} (F(x,·),u)u  \in H^s_\omega(M) +H^{2s-1} _\omega(M)\,, 
$$
where $\mathcal R_{PL} (F(x, \cdot),u)u$ is a bounded linear operator from $H^s_\omega(M)$ to $H^{2s-1}_\omega(M)$  such that for a constant $C'_s >0$
$$
\Vert \mathcal R_{PL} (F(x,\cdot),u) \Vert_{\mathcal L(H^{s}_\omega(M), H^{2s-1}_\omega(M) )}  \leq  C'_s  \, \Vert F \Vert_{C^{N_s+3}_\omega(M \times \R^2 )} (1+\Vert u \Vert_{ H^s_\omega(M)})  \,. 
$$
Moreover, the operators $Op({ \frac{\partial F(x,u)}{ \partial u}} ) \in  \mathcal L(H^{s}_\omega(M), H^{s}_\omega(M))$ and   $\mathcal R_{PL} (F(x,\cdot),u) \in  \mathcal L(H^{s}_\omega(M), H^{2s-1}_\omega(M))$  are continuously differentiable in $u\in H^s_\omega(M)$ with respect to the operator norms.
\end{proposition}

\section{ Geodesic flow on translation surfaces}

The Hamiltonian of the flat geodesic flow on a translation surface $(M, \omega)$ has the form on $M\setminus \Sigma$ in canonical coordinates:
$$
H_0(x, \xi) = \frac{\xi_1^2 + \xi_2^2}{2}  \quad \text{ for all } (x, \xi) \in M\setminus \Sigma \times \R^2\,.
$$
The coordinate-free expression of the Hamiltonian is
$$
H_0(x, v) = \frac{1}{2}  \vert \omega_x (v ) \vert^2 \,, \quad \text{ for all }  (x,v) \in TM\,.
$$
We will consider a Hamiltonian function
$$
H(x, v) =   \frac{1}{2}  \vert \omega_x (v ) \vert^2  + f(x, v) \,, \quad \text{ for all }  (x,v) \in TM\,.
$$
with $f$ a smooth function vanishing on a neighborhood of $TM \vert \Sigma$ (in fact, it is enough to assume vanishing
at $TM \vert \Sigma$  with sufficiently high order).

The tangent bundle $TM \vert \Sigma$ can be trivialized over $M\setminus \Sigma$ since the bundle has never vanishing sections
$X$ and $Y$, hence 
$$
v = \xi_1 X_1+ \xi_2 X_2 \,,
$$
and in the same coordinates
$$
H(x, \xi) = \frac{\xi_1^2 + \xi_2^2}{2}  + f(x, \xi)\,.
$$
The Hamiltonian vector field $X_H$ has the form
$$
\begin{aligned}
X_H (x,\xi)  &=  \frac{\partial H}{\partial \xi_1} X_1 +  \frac{\partial H}{\partial \xi_2} X_2 -  X_1 H \frac{\partial}{\partial \xi_1}
 -  X_2H  \frac{\partial}{\partial \xi_2} \\
&=  
\xi_1 X_1 + \xi_2 X_2 +  \frac{\partial f}{\partial \xi_1}(x,\xi) X_1 + \frac{\partial f}{\partial \xi_2}(x,\xi) X_2  \\ 
    &\quad \quad  -  X_1 f (x,\xi) \frac{\partial}{\partial \xi_1} -  X_2 f (x,\xi) \frac{\partial}{\partial \xi_2}\,. 
\end{aligned}
$$

The equation of an invariant surface is of the form 
\begin{equation}
\label{eq:inv_surf}
\mathcal{F}_\xi(H, u) := X_H \circ u - X_\xi (u) =0 \,.
\end{equation}
with $u: M \to M\times \R^2$, so that the invariant surface is $u(M) \subset M\times \R^2$.

We proceed to a (standard) computation of the differential $D_u \mathcal{F}_\xi(H, u)$.
Let 
$$
A[u] = \begin{pmatrix} D_X \nabla_\xi H (u)  &  D_\xi \nabla_\xi H (u)\\ -  D_X \nabla_X H (u)  & - D_\xi \nabla_X H (u) \end{pmatrix}  \in M_{4\times 4} (\R)\,.
$$
so that
$$
D_u \mathcal{F}_\xi(H, u) (v) = A[u] (v) - X_\xi (v) \,.
$$
Following \cite{AlSh}, \cite{LGJV05}, we introduce
$$
N[u] =  (Du^t \cdot Du)^{-1} \in M_{2\times 2} (\R) \quad \text{ and } \quad   M[u] = \begin{pmatrix}  Du &  (JDu) \cdot N[u] \end{pmatrix} \in M_{4\times 4}(\R) \,.
$$
In the above formulas $Du$ denotes the differential of $u = (u_1, u_2): M \to M \times \R^2$, with respect to the bases $\{X_1, X_2\}$ of $TM$ and $\{ \partial / \partial \xi_1,
\partial / \partial \xi_2\}$,   as a column vector:
$$
Du = \begin{pmatrix}    Du_1   \\   D u_2 \end{pmatrix}   \in  M_{4, 2} (\R)\,, \quad JDu= \begin{pmatrix} 0 & I_2 \\-I_2 & 0 \end{pmatrix}  \begin{pmatrix}    Du_1   \\   D u_2 \end{pmatrix}
=  \begin{pmatrix}    Du_2   \\   -D u_1 \end{pmatrix} \in  M_{4, 2} (\R)\,.
$$
Since $u$ is by hypotheses close to $u_0$, defined as $u_0(x)= (x, \xi)$  for $(x, \xi) \in M \times \R^2$, it follows that $N[u]$ is close to the identity, and $M[u]$, $M[u]^{-1}$
are close to $\text{diag}(I_2, -I_2)$  (with an error uniformly bounded in terms of the  uniform  norm of $Du$).

Since the translation structure on $M$ (of genus $g\geq 2$) has cone points at a finite set $\Sigma$, the space of smooth maps $u : M \to M \times \R^2$ is not locally a vector
space. However, the finite codimensional  subspace determined by the condition that the restriction
$$
u_1 \vert \Sigma = u_{0,1} \vert \Sigma = \text{Id}_\Sigma \,
$$
can be locally identified with a  ball  (with sufficiently small radius) of functions with values in $\R^2 \times \R^2$ which vanish at $\Sigma$. Indeed, for such 
functions we may have (in terms of the flat distance $d_\omega$ on the translation surface $(M,\omega)$):
$$
\text{d}_\omega ( u(x), x ) <  \text{d}_\omega ( x, \Sigma)\,, \quad \text{ for all } x \in M\,.
$$
The  linearization of the equation \eqref{eq:inv_surf} is a cohomological equation for the translation vector field  $X_\xi$. Cohomological equations
on translation surfaces were investigated in \cite{F97}, \cite{F21},  and \cite{MMY05}, \cite{MY16} for Interval Exchange Transformations (IET's), which appear 
as return maps of translation flows to transverse intervals. All of the above paper, except~\cite{F97} hold for almost all translation surfaces, in fact under a precise 
Roth-type full measure condition on the IET.  It was proved in \cite{CE15} on the basis of the Magic Wand Theorem of Eskin, Mirzakhani and Mohammadi  \cite{EM18},
\cite{EMM15} and  subsequent results of S.~Filip \cite{Fil16}, that the Roth-type condition in fact holds for every translation surface in almost all directions.

We state below the simplest form of such results, going back to \cite{F97}. 

For all $s\geq 0$, and for almost all $\xi \in \R^2$ let $\mathcal  I^s_\xi (M) \subset H^{-s}_\omega (M)$ denote the space of invariant distributions for the vector field $X_\xi$ on $M$,
that is,
$$
\mathcal  I^s_\xi (M) := \{ D \in H^{-s}_\omega (M) \vert X_\xi D =0 \}\,.
$$
\begin{theorem} \cite{F97}  \label{thm:CE} There exists $s_0 >0$ such that, for all $s >s_0$ and for almost all $\xi \in \R^2$  the cohomological
equation $X_\xi u =f$ has a solution under the following conditions:
\begin{itemize}
\item  there exists a constant $C_{s}(\xi) >0$ such that  if $f \in H^s_\omega(M)$  has zero average with respect to the area form on $(M, \omega)$ then there exists a distributional solution $u \in H^{-s_0}_\omega (M)$ such that
$$
\Vert u \Vert_{H^{-s_0}_\omega(M)} \leq C_s(\xi)    \Vert f \Vert _{H^{s}_\omega(M)} \,;
$$
\item for all $0\leq t <s-s_0$ and there exists a constant $C_{s,t}(\xi) >0$ such that if $f \in H^s_\omega(M)$ belongs to the kernel $\text{Ker} \Big(\mathcal  I^s_\xi (M) \Big)$ of the 
space of invariant distributions, which is finite dimensional, then there exists a  unique zero-average solution $u\in H^t_\omega(M)$  and the following estimate holds:
$$
\Vert u \Vert_{H^{t}_\omega(M)} \leq C_{s,t} (\xi)    \Vert f \Vert _{H^{s}_\omega(M)} \,;
$$
\end{itemize} 
\end{theorem} 

It is then immediate to derive the existence of solutions vanishing (at any finite order) at $\Sigma$ under a finite number of additional 
independent distributional  conditions.

\begin{corollary} 
\label{cor:CE}
For any $k\in \N$, there exists $s_k >0$ such that, for all $s >s_k$ and for almost all $\xi \in \R^2$,  there exists a finite dimensional
space $\mathcal  I^{s,k}_\xi (M) \subset H^{-s}_\omega(M)$ such that  the cohomological equation $X_\xi u =f$ has a solution vanishing at order $k$
on the finite set $\Sigma$ under the condition that $f \in \text{Ker}\Big( \mathcal  I^{s,k}_\xi (M) \Big)  \subset H^{s}_\omega(M)$\,.
\end{corollary} 
\begin{proof} By Theorem \ref{thm:CE}, for all $s>s_0$ and for almost all $\xi\in \R^2$,  there exists a Green operator $G^s_\xi : \text{Ker}\Big( \mathcal  I^{s}_\xi (M) \Big)  \to
H^t_\omega(M)$ for $t<s-s_0$  (with values in the subspace of zero average functions). The condition of vanishing at $\Sigma$ at order $k$ is
given by a finite number of distributions supported at $\Sigma$, which are derivatives of Dirac masses at $\Sigma$, with Sobolev order up to 
$k+1$ (by the Sobolev embedding theorem). Let us assume then that $s_k> s_0 + k +1$. By Theorem  \ref{thm:CE} the compositions $\delta_\Sigma^{(j)} \circ G_\xi^s$,
of derivatives $\delta_\Sigma^{(j)}$ of order $j\leq k$ of Dirac masses at $\Sigma$ with the Green operator $G^s_\xi$, give well-defined distributions on 
$H^{s_k}_\omega(M)$, since $s_k -s_0  > k+1$. 

By definition, if $f\in H^s_\omega(M)$ with $s>s_k$ belongs to 
the kernel of $\mathcal  I^{s}_\xi (M)$ and to that of all the additional distributions $\delta_\Sigma^{(j)} \circ G_\xi^s$, then the unique zero-average solution
$u\in H^t_\omega(M)$ of the equation $X_\xi u=f$ vanishes at order $k$ on $\Sigma$.

\end{proof} 

\smallskip
\noindent It is well known that invariant submanifolds of Hamiltonian flows should be expected to be Lagrangian.  In our context
for any differentiable map $u: M \to \R^2$,  the function
$$
L[u]=  (Du)^t J Du \,,
$$
measures how for $u$ is from being a Lagrangian embedding.   In fact $L[u]=0$ if and only if the differential of the pull-back  under 
$u$ of the standard symplectic form on $M \times \R^2$ is a closed $2$-form. In addition $L[u] =0$ whenever
$\mathcal{F}_\xi(H, u)=0$ ($u$ is an invariant section) and the Sobolev norms of $L[u]$ can be bounded in terms of those of $\mathcal{F}_\xi(H, u)=0$,
as in the following result.

\begin{lemma} (see \cite[Lemma 19]{LGJV05}) \label{lemma:L} If the translation flow in the direction $\xi \in \R^2\setminus \{0\}$ is (quasi)-minimal  on $M$, 
then the function $L[u]$ depends linearly on $\mathcal{F}_\xi(H, u)$ for all continuously differentiable $u:M \to \R^2$.
In fact, we have the formula
$$
X_\xi L[u]  = - \Big[D \mathcal{F}_\xi(H, u)^t J Du + (Du)^t J D\mathcal{F}_\xi(H, u) \Big] \,.
$$
As a consequence, for almost all $\xi \in \R^2$, for every $s>s_0$  and any $t<s-s_0$ there exists $C_{s,t}(\xi) >0$ such that, 
whenever $u \in H^{s+2}_\omega (M)$, we have 
$$
\Vert L[u] \Vert _{H^t_\omega(M)}  \leq   C_{s,t}(\xi)  \Vert Du \Vert_{H^s_\omega(M)} \Vert D\mathcal{F}_\xi(H, u)  \Vert_{H^s_\omega(M)} \,.
$$
\end{lemma} 
\begin{proof} 
Since by definition $\mathcal{F}_\xi(H, u)  = X_H \circ u - X_\xi u$  and $D (X_H \circ u) = A[u] Du$, by the Hamiltonian character of the equations
$ A[u]^t J  + J A[u] =0$ we can compute as follows:
$$
\begin{aligned}
X_\xi L[u] &= (D X_\xi u)^t J Du +  (Du)^t J D X_\xi u  \\ & =  \Big[D(X_H\circ u - \mathcal{F}_\xi(H, u) )\Big]^t J Du + 
(Du)^t J D \Big [ (X_H \circ u -\mathcal{F}_\xi(H, u)\Big] 
 \\ & =   (Du)^t  \Big ( A[u]^t J + J A[u] \Big) Du -  D \mathcal{F}_\xi(H, u)^t J Du + (Du)^t J D\mathcal{F}_\xi(H, u) \Big]
 \\& = - \Big[D \mathcal{F}_\xi(H, u)^t J Du + (Du)^t J D\mathcal{F}_\xi(H, u) \Big]     \,.
\end{aligned}
$$
Since $L[u]$ has zero average on $M$ (by its definition and integration by parts), the estimates on the Sobolev norms of $L[u]$ follow from the above equation and Theorem~\ref{thm:CE}. In fact, the second part
of  the theorem is equivalent to the following a priori estimate: for $s>s_0$ and $t<s-s_0$ there exists $C_{s,t}(\xi)>0$ such that, for all $v \in H^{s+1}_\omega(M)$ of zero average,
\begin{equation}
\label{eq:a_priori_est}
\Vert  v \Vert_{H^t_\omega(M)} \leq C_{s,t}(\xi) \Vert X_\xi v \Vert_{H^s_\omega(M)}.
\end{equation}
The above a priori estimate can be derived from Theorem~\ref{thm:CE} as follows. For any $v \in H^\infty_\omega (M) = \cap_{s>0} H^s_\omega(M)$,
the function $X_\xi v \in \text{ker}(\mathcal I^s_\xi(M))$ for all $s >0$.  From Theorem \ref{thm:CE} it follows the the a priori estimate  \eqref{eq:a_priori_est}
holds. The a priori estimate can then be extended by continuity, since $H^\infty_\omega (M)$ is dense in $H^s_\omega(M)$ for all $s>0$.

\end{proof}

\section{Para-linearization}

We proceed to compute the para-linearization of the non-linear equation of an invariant surface.

\begin{lemma} 
\label{lemma:lin_id} 
(see  \cite{LGJV05}, Lemma 20) The following identities for the linearization of $\mathcal{F}_\xi(H, u)$ hold:
$$
D_u \mathcal{F}_\xi(H, u) (M[u] v)  =  M[u] \begin{pmatrix} 0_2 & S[u] \\ 0_2 & 0_2 \end{pmatrix} v - M[u] X_\xi v + B[u,L[u],\mathcal{F}_\xi(H, u)] v\,. 
$$
In the above formula we have
$$
\begin{aligned}
S[u]:=  (I_2 + (N[u] L[u])^2)^{-1} N[u] \cdot (Du)^t  \cdot [A[u], J] \cdot Du \cdot  N[u]  \in M_{2\times 2}(\R) \,.
\end{aligned}
$$
Finally, the following crucial property holds:  the term $B[u, L[u], \mathcal{F}_\xi(H, u)] $ is rational in $L[u]$, linear in  $\mathcal{F}_\xi(H, u)$, 
and vanishes for $ L[u] = \mathcal{F}_\xi(H, u)=0$, in fact 
$$
B [u, L[u], \mathcal{F}_\xi(H, u)]  = \begin{pmatrix} B_1[\mathcal{F}_\xi(H, u)]   & B_2[u, L[u], \mathcal{F}_\xi(H, u)]  \end{pmatrix} 
$$
where for any mapping $L, E : M \to M_{2\times 2}(\R)$ we have 
$$
\begin{aligned}
B_1[E] &:= DE \,, \\  
B_2[u,L, E]  &:= (Du) (I_2 + (N[u] L)^2)^{-1} \\ & \qquad \times \Big\{N[u]L N[u]  [DE^t  \cdot Du \cdot  N[u]    + N[u] (Du)^t J DE  \cdot N[u]  - N[u]L X_\xi N[u]  \Big\}  \\ &+ J (Du) N[u]  (I_2 + (L N[u])^2)^{-1} \\
&   \qquad \times \Big\{ -DE^t  \cdot Du \cdot  N[u]  + L N[u] (Du)^t  [A[u], J]  Du \cdot  N[u]  \\ & \qquad \qquad+ L N[u] (Du)^t J D\mathcal{F}_\xi(H, u)  \cdot N[u]    - LN[u]L X_\xi N[u]  \Big \} \,.
\end{aligned} 
$$
We may then rewrite the above relation in the equivalent form
$$
D_u \mathcal{F}_\xi(H, u) (v)  =  M[u] \begin{pmatrix} 0_2 & S[u] \\ 0_2 & 0_2 \end{pmatrix} M[u]^{-1} v - M[u] X_\xi (M[u]^{-1} v) + B[u,L[u], \mathcal{F}_\xi(H, u)    ] M[u]^{-1}v\,. 
$$
\end{lemma}

\begin{proof}
Let $\mathcal{F}_\xi(H, u) := X_H \circ u - X_\xi (u)$. It follows that
$$
D_u \mathcal{F}_\xi(H, u) ( M[u] v) = A[u] (M[u] v) - (X_\xi M[u] ) v -  M[u] (X_\xi v)\,.
$$
We compute the matrix $A[u] M[u] - X_\xi M[u]$.  We compute the first two columns.  Since for $M[u]$ they are given by $Du$ and we have
$$
D\mathcal{F}_\xi(H, u) := A[u] Du - X_\xi (Du)\,,
$$
it follows that the first two columns of $A[u] M[u] - X_\xi M[u]$ are equal to $D\mathcal{F}_\xi(H, u)$, hence $B_1[E]= DE$ as stated.

\medskip
We then compute the last two columns. The last two columns for $M[u]$ are $(JDu) \cdot N[u]$. We therefore compute
$$
A[u]  \cdot J Du \cdot  N[u] -  X_\xi (JDu \cdot N[u])\,.
$$ 
Since $Du$ and $J Du$ form a basis of $\R^2$ and $N[u]$ is invertible (for $u$ near $u_0$) there exist matrices
$P$ and $Q \in M_{2\times 2}(\R)$ such that we can write
$$
A[u]  \cdot J Du \cdot  N[u] -  X_\xi (JDu \cdot N[u]) =  (Du) P  + J (Du) N[u]  Q \,.
$$
Since $ (Du)^t  J Du = L[u]$, $(Du)^t (Du) N =I_2$,  $J^2 =-I_4$ and $ J A[u] = - A[u]^t J$, we have
$$
\begin{aligned}
-L[u] P + Q &= -Du^t  J \Big( A[u]  \cdot J Du \cdot  N[u] -  X_\xi (J Du \cdot N[u]) \Big) \\
&=  - Du^t   (A[u])^t  Du \cdot  N[u]  - Du^t  X_\xi (Du \cdot N[u]) \\
&=  - \Big(  X_\xi (Du)^t + [D\mathcal{F}_\xi(H, u)]^t \Big) \cdot Du \cdot  N[u]  - (Du)^t  X_\xi (Du \cdot N[u]) \\
& = - X_\xi  \Big( (Du)^t  \cdot Du \cdot  N[u] \Big) -   [D\mathcal{F}_\xi(H, u)]^t  \cdot Du \cdot  N[u] \\
&=   - [D\mathcal{F}_\xi(H, u)]^t  \cdot Du \cdot  N[u]\,.
\end{aligned}
$$
We also have, since $ (Du)^t  J Du =L[u]$,
$$
\begin{aligned}
P + &N[u] L[u] N[u] Q=  N[u] (Du)^t \Big( A[u]  \cdot J Du \cdot  N[u] -  X_\xi (J Du \cdot N[u]) \Big) \\
&=  N[u] (Du)^t \Big( A[u]  \cdot J Du \cdot  N[u] -  X_\xi (J Du) \cdot N[u] \Big) - N[u]L[u] X_\xi N[u] \\
&=  N[u] (Du)^t \Big( A[u]  \cdot J Du \cdot  N[u] -     J \Big( A[u](Du)-  D\mathcal{F}_\xi(H, u) \Big)    \cdot N[u] \Big) - N[u]L[u] X_\xi N[u] \\
&= N[u] (Du)^t  [A[u], J]  Du \cdot  N[u]  + N[u] (Du)^t J D\mathcal{F}_\xi(H, u)  \cdot N[u]  - N[u]L[u] X_\xi N[u] \,.
\end{aligned}
$$
Thus we have the formulas
$$
\begin{cases} P = & (I_2 + (N[u] L[u])^2)^{-1}\Big\{N[u]L[u] N[u]  [D\mathcal{F}_\xi(H, u)]^t  \cdot Du \cdot  N[u]   \\ & + N[u] (Du)^t  [A[u], J]  Du \cdot  N[u]  + N[u] (Du)^t J D\mathcal{F}_\xi(H, u)  \cdot N[u]  - N[u]L[u] X_\xi N[u]  \Big\} \,;   \\
 Q =  & (I_2 + (L[u] N[u])^2)^{-1}\Big\{-D\mathcal{F}_\xi(H, u)^t  \cdot Du \cdot  N[u]  + L[u] N[u] (Du)^t  [A[u], J]  Du \cdot  N[u]    \\ &  \quad + L[u] N[u] (Du)^t J D\mathcal{F}_\xi(H, u)  \cdot N[u]   - L[u] N[u]L[u] X_\xi N[u]  \Big\} \,.
\end{cases}  
$$
We conclude that  the stated identity holds with
$$
\begin{aligned} &S[u]:=  (I_2 + (N[u] L[u])^2)^{-1} N[u] \cdot (Du)^t  \cdot [A[u], J] \cdot Du \cdot  N[u] \} \,,  \\
                        &B_2[u]  := (Du) (I_2 + (N[u] L[u])^2)^{-1}\Big\{N[u]L[u] N[u]  [D\mathcal{F}_\xi(H, u)]^t  \cdot Du \cdot  N[u]    \\ & \qquad \qquad  + N[u] (Du)^t J D\mathcal{F}_\xi(H, u)  \cdot N[u]  - N[u]L[u] X_\xi N[u]  \Big\}  \\ &+ J (Du) N[u]  (I_2 + (L[u] N[u])^2)^{-1} \Big\{ -D\mathcal{F}_\xi(H, u)^t  \cdot Du \cdot  N[u]  \\ &  + L[u] N[u] (Du)^t  [A[u], J]  Du \cdot  N[u] + L[u] N[u] (Du)^t J D\mathcal{F}_\xi(H, u)  \cdot N[u]    - L[u]N[u]L[u] X_\xi N[u]  \Big \} \,.
\end{aligned} 
$$
\end{proof} 

The para-linearization formula of Proposition \ref{prop:para-lin} gives
$$
\mathcal{F}_\xi(H, u)  = \mathcal{F}_\xi(H, u_0) + T_{D_u \mathcal{F}_\xi(H, u)} (u-u_0) + \mathcal R_{PL} (\mathcal{F}_\xi(H, u) ,  u-u_0) (u-u_0)\,.
$$
 Let then
$$
E := \mathcal{F}_\xi(H, u) \,.
$$
By the above lemma we have
$$
\begin{aligned}
&T_{D_u \mathcal{F}_\xi(H, u)} (u-u_0)  =  Op \Big ( M[u] \begin{pmatrix} 0_2 & S[u] \\ 0_2 & 0_2 \end{pmatrix} M[u]^{-1} \Big) (u-u_0) \\ &   - 
T_{ M[u] } X_\xi \Big(T_{M[u]^{-1} }  (u-u_0) \Big)  + Op \Big( B[u,L, E] (M[u])^{-1} \Big) (u-u_0) + \mathcal R'_{CM}[u] (u-u_0) \,, 
\end{aligned}
$$
hence the para-linearization formula can be written as follows:
$$
\begin{aligned}
E&= \mathcal{F}_\xi(H, u_0)   + T_{M[u]}  \begin{pmatrix} 0_2 & T_{S[u]} \\ 0_2 & 0_2 \end{pmatrix} T_{M[u]^{-1}}  (u-u_0) \\ & \quad  - 
T_{ M[u] } X_\xi \Big(T_{M[u]^{-1} }  (u-u_0) \Big) + Op \Big( B [u,E] (M[u])^{-1} \Big) (u-u_0)  \\ &
\quad + 
\mathcal R_{PL}[E, u-u_0 ] (u-u_0) + \mathcal R_{CM}[u] (u-u_0) \,.
\end{aligned} 
$$
The above formula leads to the para-differential (co)homological equation
$$
\begin{aligned}
&  T_{M[u]} \begin{pmatrix} 0_2 & T_{S[u]} \\ 0_2 & 0_2 \end{pmatrix} T_{M[u]^{-1}}  (u-u_0)  - 
T_{ M[u] } X_\xi \Big(T_{M[u]^{-1} }  (u-u_0) \Big)   \\ & \quad    
 =   -\mathcal{F}_\xi(H, u_0) -\mathcal R_{PL}[E, u-u_0 ] (u-u_0) - \mathcal R_{CM}[u] (u-u_0)    \,.
\end{aligned}
$$

\section{Solution of the para-differential cohomological equation}

We write the cohomological equation with counter-terms. Let $\{\chi_i\}$ be a dual basis of the space of invariant distributions for the vector field $X_\xi$. 
Let $[\chi_i]$ denote the $4\times 2$ matrix with entries all equal to $\chi_i$ and $c_i$ a constant diagonal $4\times 4$ matrix.  Let  $\mathcal T[u]$
denote a bounded linear operator (to be determined).
\begin{equation}
\label{eq:para-cohom}
\begin{aligned}
&  T_{M[u]} \begin{pmatrix} 0_2 & T_{S[u]} \\ 0_2 & 0_2 \end{pmatrix} T_{M[u]^{-1}}  (u-u_0)  - 
T_{ M[u] } X_\xi \Big(T_{M[u]^{-1} }  (u-u_0) \Big)  + \mathcal T[u] \Big(  \sum_i  c_i [\chi_i]  \Big)  \\ &    
    \qquad = -\mathcal{F}_\xi(H, u_0) -\mathcal R_{PL}[E, u-u_0 ] (u-u_0) - \mathcal R_{CM}[u] (u-u_0)    \,.
\end{aligned}
\end{equation}

At this point we prove existence of solutions (that is, invertibility of the operator) for the equation
$$
 T_{M[u]} \begin{pmatrix} 0_2 & T_{S[u]} \\ 0_2 & 0_2 \end{pmatrix} T_{M[u]^{-1}}  v - 
T_{ M[u] } X_\xi \Big(T_{M[u]^{-1} }  v \Big)  + \mathcal T[u] \Big(  \sum_i  c_i [\chi_i]  \Big) = f
$$
Since $M[u]= \text{diag}(I_2, -I_2) + O( \Vert u -u_0 \Vert_{C^1_\omega(M)})$, it is an invertible matrix for $\Vert u -u_0 \Vert_{C^1_\omega(M)}$ small enough. 
Under this hypothesis we can write the equation in the form
$$
 \begin{pmatrix} 0_2 & T_{S[u]} \\ 0_2 & 0_2 \end{pmatrix} T_{M[u]^{-1}}  v - 
 X_\xi \Big(T_{M[u]^{-1} }  v \Big)  +  T_{M[u]}^{-1} \mathcal T[u] \Big(  \sum_i  c_i [\chi_i]  \Big) =  T_{M[u]}^{-1} f
$$
hence with the choice of the linear operator $\mathcal T[u] = T_{M[u]}$ we have the equation
$$
\begin{pmatrix} 0_2 & T_{S[u]} \\ 0_2 & 0_2 \end{pmatrix} T_{M[u]^{-1}}  v - 
 X_\xi \Big(T_{M[u]^{-1} }  v \Big)  +   \sum_i  c_i [\chi_i]   =  T_{M[u]}^{-1} f
$$
The existence (with bounds) of solution to the above equation  is proved below.

\begin{lemma}  
\label{lemma:solutions}
Let $\xi \in \R^2 \setminus \{0\}$ be such that the translation vector field $X_\xi$ on the translation surface $(M, \omega)$ is stable
(that is, the Lie derivative operator has close range) of finite codimension in Sobolev spaces of finitely differentiable functions, with loss of $\sigma>0$
derivatives in the scale $\{H^s_\omega(M) \}$ of weighted Sobolev spaces on $(M, \omega)$.
There are constants $\rho_1$, $\rho_2>0$ depending on $\Vert H \Vert_{C^3_\omega(M)}$  and a constant $K>0$ with the following property. If  $\Vert \mathcal F (H, u_0)\Vert_{C^1_\omega(M)} \leq \rho_1$, and the embedding $u: M \to M \times \R^2$  is such that $u\vert \Sigma = u_0 \vert \Sigma= \text{Id}_\Sigma$ and $\Vert u- u_0\Vert_{C^1_\omega(M)} \leq \rho_2$, then the linear para-homological equation in the unknown $(v, c)$,
$$
T_{M[u]} \begin{pmatrix} 0_2 & T_{S[u]} \\ 0_2 & 0_2 \end{pmatrix} T_{M[u]^{-1}}  v - 
T_{ M[u] } X_\xi \Big(T_{M[u]^{-1} }  v \Big)  + T_{M[u]}\Big( \sum_i  c_i [\chi_i] \Big) = f \,,
$$
has a linear solution operator 
$$
(v, c) = (v_1, v_2, c_1, c_2) =  (\mathcal L[u] (f),  P[u] (f)) =(\mathcal L_1[u] (f), \mathcal L_2[u] (f),  P_1[u] (f), P_2[u] (f) )
$$
such that the range of the operator  $\mathcal L[u]$ is contained in the subspace of functions vanishing at $\Sigma$, the range of the operator $P$
is finite dimensional, and the following estimate holds: for a function $\Phi$ increasing in all of its arguments, we have
$$
\Vert \mathcal L [u] (f) \Vert_{H^s_\omega(M)} + \vert P[u] (f) \vert \leq  C_s (K, \Vert H\Vert_{C^3_\omega(M)}, \xi) \Vert  f \Vert_{H^{s+2\sigma}_\omega(M)} \,.
$$
Moreover, the four linear operators of concern are all continuously differentiable mappings from $u\in C^1_\omega(M)$ to the space of linear operators (with operator norm). 
\end{lemma} 
\begin{proof}
We remark that by the definition of the matrix $S[u]$ we have
$$
\Vert S[u]  +I_2\Vert = \Vert S[u] - S[u_0] \Vert \leq C \Vert H \Vert_{C^3_\omega(M)} \Vert u-u_0 \Vert_{C^1_\omega(M)} \,.
$$
In fact, in our case we have that, 
$$
A[u_0] = \begin{pmatrix}   0_2   &   I_2 \\  0_2 & 0_2  \end{pmatrix} 
$$
and $Du_0^t =\begin{pmatrix} I_2& 0_2 \end{pmatrix}$, hence $L[u_0] =0$ and thus
$$
S[u_0] = \begin{pmatrix}  I_2 & 0_2 \end{pmatrix} \begin{pmatrix}   -I_2   &   0_2 \\  0_2 & I_2  \end{pmatrix}    \begin{pmatrix}  I_2 \\ 0_2 \end{pmatrix} = -I_2\,. 
$$
There exists $\rho_2>0$ such that for $\Vert u-u_0 \Vert_{C^1_\omega(M)} \leq \rho_2$, the para-products  $T_{M[u]}$ and $T_{M[u]^{-1}}$ are invertible with inverse 
such that 
$$  \Vert T^{-1}_{M[u]} -I \Vert_{H^s_\omega(M)} +  \Vert T^{-1}_{M[u]^{-1}} -I \Vert_{H^s_\omega(M)}   \leq 1/2\,.
$$
The equation can then be written as a system of two systems of 2 equations equations for the unknown vector-valued function 
$\hat v := T_{M[u]^{-1}}  v$:
$$
\begin{cases}  
T_{S[u]}  \hat v_2 -  X_\xi \hat v_1  + \sum_i  c_{i,1} [\chi_{i}]_1  =  (T^{-1}_{M[u]}  f )_1 \,, \\ \quad\quad \quad
-  X_\xi \hat v_2  + \sum_i  c_{i,2} [\chi_{i}]_2  =  (T^{-1}_{M[u]}  f )_2\,,
 \end{cases}
$$
This system can be solved by first solving the second equation, for an appropriate choice of the constants $\{c_{i,2}\}$. The solution
$\hat v_2$ which unique up to additive constants, can be plugged in the first equation, which can then be solved for an appropriate choice 
of the constants $\{c_{i,1}\}$.  More precisely the equation
$$
-  X_\xi \hat v_2  + \sum_i  c_{i,2} [\chi_{i}]_2  =  (T^{-1}_{M[u]}  f )_2
$$
can be solved in $H^{s+\sigma}_\omega(M)$  if (and only if) $c_{i,2} = D_i \Big( (T^{-1}_{M[u]}  f )_2\Big)$, for all invariant distributions 
$D_i$ in $H^{-(s+ 2\sigma)}_\omega(M)$, hence there exists a constant $C_{2,s}(\xi) >0$ such that 
$$
\sum_i \vert c_{i,2} \vert  \leq  C_{2,s} (\xi)  \Vert u \Vert_{C^1_\omega(M)} \Vert f \Vert_{H^{s+2\sigma}_\omega(M)}\,.
$$
The solution satisfies the estimate 
$$
\begin{aligned}
\Vert  v_2 \Vert_{H^{s + \sigma}_\omega(M)} &\leq  C'_{2,s} (\xi)  \Vert (T^{-1}_{M[u]}  f )_2- \sum_i  c_{i,2} [\chi_{i}]_2\Vert_{H^{s + 2\sigma}_\omega(M)}
\\ & \leq  C''_{2,s} (\xi)  \Vert u \Vert_{C^1_\omega(M)}  \Vert f \Vert_{H^{s + 2\sigma}_\omega(M)} \,.
\end{aligned}
$$
The first equation can be then solved in $H^s_\omega(M)$  if 
$$
c_{i,1} = D_i \Big (  (T^{-1}_{M[u]}  f )_1 - T_{S[u]}  \hat v_2\Big)\,,
$$
for all invariant distributions 
$D_i$ in $H^{-(s+ \sigma)}_\omega(M)$, hence there exists a constant $C_{1,s}(\xi) >0$ such that 
$$
\sum_i \vert c_{i,2} \vert \leq \Vert  (T^{-1}_{M[u]}  f )_1 - T_{S[u]}  \hat v_2   \Vert_{H^{s+\sigma}_\omega(M)}  \leq  C_{1,s} (\xi)  \Vert u \Vert_{C^1_\omega(M)} \Vert f \Vert_{H^{s+2\sigma}_\omega(M)}\,.
$$
The solution of the first equation then satisfies the bound
$$
\begin{aligned}
\Vert  v_1 \Vert_{H^{s}_\omega(M)} &\leq  C'_{2,s} (\xi)  \Vert  (T^{-1}_{M[u]}  f )_1 - T_{S[u]}  \hat v_2-  \sum_i  c_{i,1} [\chi_{i}]_1\Vert_{H^{s + \sigma}_\omega(M)}
\\ & \leq  C''_{2,s} (\xi)  \Vert u \Vert_{C^1_\omega(M)}  \Vert f \Vert_{H^{s + 2\sigma}_\omega(M)} \,.
\end{aligned}
$$
Finally, continuous differentiability of the operators for $u\in C^1_\omega(M)$  follows immediately from properties of para-products, that is, continuous differentiability of 
$T_a$ in $a$ and $T_{ab}- T_aT_ b$ with respect to $(a, b)\in (L^{\infty})^2$.
\end{proof}

\section{Conclusion and open questions}

In this section we proceed to prove the main theorem.

\begin{proof} [Proof of Theorem~\ref{thm:main}]
Let  $N_s$ be the smallest integer $>2s-1$. Let $H \in C^{N_s+4}_\omega(M)$ so that the Hamiltonian vector field $X_H$ has coefficients in $C^{N_s+3}_\omega(M)$.  We can assume $\Vert u \Vert_{H^s_\omega(M)} \leq 1$, that  $\rho_1$, $\rho_2>0$ are as in Lemma \ref{lemma:solutions}
and also assume that the hypotheses of Lemma \ref{lemma:solutions} are satisfied, that is
\begin{equation}
\label{eq:apriori_conditions}
\Vert u \Vert_{H^s_\omega(M)} \leq 1\,, \quad  \Vert \mathcal F (H, u_0) \Vert_{C^1_\omega(M)} \leq \rho_1\,, \quad  \Vert u-u_0\Vert _{C^1_\omega(M)} \leq \rho_2 \,.
\end{equation}
The para-cohomological equation has the form
$$
u-u_0 = - \mathcal{L}[u] \Bigl (\mathcal{F}_\xi(H, u_0) +  \mathcal R_{PL}[\mathcal F (H, u), u-u_0 ] (u-u_0) + \mathcal R_{CM}[u] (u-u_0)   \Big)
$$
By Lemma~\ref{lemma:solutions}  we have
$$
\Vert  \mathcal{L}[u] \Bigl (\mathcal{F}_\xi(H, u_0) \Big) \Vert_{H^s_\omega(M)} \leq C_s (K, \Vert H\Vert_{C^3_\omega(M)}, \xi) \Vert  \mathcal{F}_\xi(H, u_0)  \Vert_{H^{s+2\sigma}_\omega(M)} 
$$
We also have, for $2t-1 >  s+ 2\sigma$, 
$$
\begin{aligned}
\Vert \mathcal{L}[u] &\Bigl ( \mathcal R_{PL}[\mathcal F (H, u), u-u_0 ] (u-u_0) \Big)\Vert_{H^s_\omega(M)}  \\&\quad \leq C_s (K, \Vert H\Vert_{C^3_\omega(M)}, \xi) \Vert 
\mathcal R_{PL}[\mathcal F (H, u), u-u_0 ] (u-u_0) \Vert_{H^{s+ \sigma}_\omega(M)}  \\ &\quad \leq  C_s (K, \Vert H\Vert_{C^3_\omega(M)}, \xi) C_s 
\Big(  \Vert  \mathcal{F}_\xi(H, u_0)  \Vert_{H^{s+2\sigma}_\omega(M)}  \Vert u-u_0 \Vert_{H^t_\omega(M)}  \\ & \qquad \qquad\qquad\qquad \qquad\qquad \qquad+ \Vert H \Vert_{C^{N_s+4}_\omega(M)}  \Vert u-u_0 \Vert_{H^t_\omega(M)} ^2\Big)\,.
\end{aligned}
$$
The regularizing remainder  $\mathcal R_{CM}[u] (u-u_0)$ is equal to the expression 
$$
\begin{aligned}
\mathcal R_{CM}[u] (u-u_0)&=  Op \Big( M[u] \begin{pmatrix} 0_2 & S[u] \\ 0_2 & 0_2 \end{pmatrix} M[u]^{-1}  \Big) \\ & -  Op\big(M[u] X_\xi M[u]^{-1} \Big)
- T_{M[u]} \begin{pmatrix} 0_2 & T_{S[u]} \\ 0_2 & 0_2 \end{pmatrix} T_{M[u]^{-1}}  + T_{M[u]} X_\xi T_{M[u]^{-1}}\,.
\end{aligned}
$$
Since  for $s >t > 1+ (s+\sigma)/2$  we have that $ \Vert S[u] \Vert_{ C_\omega^{s+ \sigma-t}(M) }$  is bounded by an increasing function for $\Vert H \Vert_{ H^{N_s}_\omega(M)}$ and  $\Vert u-u_0 \Vert_{ H^t_\omega(M)}$,   and  in addition
$$
\Vert M[u] - \textrm{diag} (I_2, -I_2)  \Vert_{ C_\omega^{s+ \sigma-t}(M) } \leq \Vert u-u_0 \Vert_{ H^t_\omega(M)}\,,
$$
 we derive the estimate
$$
\Vert \mathcal R_{CM}[u] (u-u_0) \Vert_{ H^{s+ \sigma}_\omega(M)} \leq  C'_s \Vert u- u_0 \Vert^2_ { H^t_\omega(M)}  \,,
$$
which in turn implies by Lemma  \ref{lemma:solutions} that 
$$
\begin{aligned}
\Vert \mathcal{L}[u] \Bigl ( \mathcal R_{CM}[u] (u-u_0) \Big)\Vert_{H^s_\omega(M)}  & \leq C_s (K, \Vert H\Vert_{C^3_\omega(M)}, \xi) \Vert 
\mathcal R_{CM}[u] (u-u_0)\Vert_{H^{s+ \sigma}_\omega(M)} \\ & \leq C_s (K, \Vert H\Vert_{C^3_\omega(M)}, \xi)  C'_s \Vert u- u_0 \Vert^2_ { H^t_\omega(M)}\,.
\end{aligned}
$$
We then argue that under conditions \eqref{eq:apriori_conditions},  if  $\Vert  \mathcal{F}_\xi(H, u_0)  \Vert_{H^{s+2\sigma}_\omega(M)}$ is sufficiently small, there exists $\rho>0$ such that if $\Vert u- u_0 \Vert_ { H^t_\omega(M)} \leq \rho$, then
$$
\begin{aligned}
C_s &(K, \Vert H\Vert_{C^3_\omega(M)}, \xi)  \Vert  \mathcal{F}_\xi(H, u_0)  \Vert_{H^{s+2\sigma}_\omega(M)}   \\ &+  C_s (K, \Vert H\Vert_{C^3_\omega(M)}, \xi) C_s \Big( \Vert  \mathcal{F}_\xi(H, u_0)   \Vert_{H^{s+2\sigma}_\omega(M)}  \rho +  \Vert H \Vert_{C^{N_s+4}_\omega(M)} \rho^2\Big)  \\ & \qquad \qquad+ C_s (K, \Vert H\Vert_{C^3_\omega(M)}, \xi)  C'_s  \rho^2 \leq \rho\,.
\end{aligned}
$$
Let then $\mathcal  S$ denote the operator defined  for all functions $u\in B_{H^t_\omega(M)} (u_0, \rho)$, that is, such that $u \vert\Sigma = u_0 \vert \Sigma= \text{Id}_\Sigma$ 
and $\Vert u -u_0 \Vert_{H^t_\omega(M)} < \rho$ as
$$
\mathcal S(u) := u_0 - \mathcal{L}[u] \Bigl (\mathcal{F}_\xi(H, u_0) +  \mathcal R_{PL}[\mathcal F (H, u), u-u_0 ] (u-u_0) + \mathcal R_{CM}[u] (u-u_0)   \Big)\,.
$$
We have proved above that $\mathcal S: B_{H^t_\omega(M)} (u_0, \rho) \to B_{H^s_\omega(M)} (u_0, \rho)$ with $s>t$, and since the embedding 
$H^s_\omega(M)$ into $H^t_\omega(M)$ is compact, by the Schauder fixed point theorem the operator $\mathcal S$ has a fixed point $u \in B_{H^t_\omega(M)} (u_0, \rho)$.
In fact, for sufficiently small  $\rho>0$ it is possible to apply the inverse function theorem in the Hilbert space $H^s_\omega(M)$, hence the fixed point is unique
and depends smoothly with respect to the Hamiltonian $H$.  In addition, there exists a constant $C(\xi, \Vert H \Vert_{C^{N_s +4}_\omega(M)}) >0$ such that
$$
\Vert u- u_0 \Vert_{H^t_\omega(M)} \leq C(\xi, \Vert H \Vert_{C^{N_s +4}_\omega(M)}) \cdot \Vert \mathcal F_\xi (H, u)  \Vert_{H^{s+2\sigma}_\omega(M)} \,.
$$

\smallskip
For the fixed point we have the identity
$$
\mathcal F_\xi (H, u) = Op \Big ( B[u, L[u], \mathcal F_\xi (H, u)] M(u)^{-1}  \Big) (u-u_0)  +  T_{M[u]} \Big(\sum_i  P_i[u] \chi_i  \Big)\,.
$$
Since the fixed point is given as a smooth function $u:=u(H)$ of the Hamiltonian $H$, defined locally on a neighborhood of $H_0$, the condition
$$
P[ u(H)] =0 \,,
$$ 
describes a finite codimension local submanifold of the space of Hamiltonians, as soon as we can prove that the differential of the function
$P[u(H)]$  (as a function of the Hamiltonian $H$) is surjective at $H=H_0$. 

For Hamiltonians $H$ such that $P[u(H)] =0$ the fixed point equation becomes
$$
\mathcal F_\xi (H, u) = Op \Big ( B[u, L[u], \mathcal F_\xi (H, u)] M(u)^{-1}  \Big) (u-u_0)  
$$
which implies $\mathcal F_\xi (H, u) =0$ by the contraction mapping principle (as in \cite{AlSh}).  Indeed, the right hand side can be viewed
(for fixed $u\in B_{H^t_\omega(M)} (u_0, \rho)$) as a contraction  operator $\mathcal B_u$ (with respect for instance to the $H^t_\omega(M)$ norm)  acting on
 $\mathcal F_\xi (H, u)$ and vanishing for $\mathcal F_\xi (H, u)=0$. To prove that $\mathcal B_u$ is a contraction we estimate
 $$
\begin{aligned} &\Vert Op \Big ( B[u,L, E] M(u)^{-1}  \Big) (u-u_0) \Vert_{H^t_\omega(M)}   \\ &\quad \leq    C_s C(\xi, \Vert H \Vert_{C^{N_s+4}_\omega(M)}) 
 \Vert \mathcal F_\xi (H, u_0) \Vert_{H^{s+2\sigma}_\omega(M)} \Vert  B[u, L[u], \mathcal F_\xi (H, u)] M(u)^{-1}   \Vert_{L^\infty(M)}    \,.
\end{aligned}
$$
By the expression of $B(u,L, E)$ in Lemma~\ref{lemma:lin_id}, there exists a constant $C_B >0$ such that, under the hypothesis that 
$\Vert u-u_0 \Vert_{C^2_\omega(M)}$ and $ \Vert L \Vert_{C^0_\omega(M)}$ are sufficiently small, we have 
$$
\Vert B(u,L, E) \Vert_{L^\infty(M)} \leq  C_B  \max \Big\{  \Vert L \Vert_{C^0(M)},  \Vert E\Vert_{C^1_\omega(M)} \big\} \,.
$$
In addition, by Lemma \ref{lemma:L}, there exists $t_0>0$ such that for $t>t_0$ and for a constant $C_t(\xi)>0$  we have
$$
\Vert L[u] \Vert _{C^0(M)}  \leq   C_{t}(\xi)  \Vert u \Vert_{H^t_\omega(M)} \Vert D\mathcal{F}_\xi(H, u)  \Vert_{H^t_\omega(M)} \,,
$$
hence, as $ \Vert u \Vert_{H^t_\omega(M)}$ is bounded, there exists  a constant $C'_t(\xi)>0$ such that
$$
\Vert B[u,L[u], \mathcal F_\xi (H, u) ] M(u)^{-1} \Vert_{L^\infty(M)}  \leq C_t(\xi) \Vert \mathcal F_\xi (H, u) \Vert_{H^t_\omega(M)}\,.
$$
Thus, if $C_s C_t C(\xi, \Vert H \Vert_{C^{N_s+4}}) \cdot
 \Vert \mathcal F_\xi (H, u_0) \Vert_{H^{s+2\sigma}_\omega(M)}  <1$ we have that $\mathcal B_u$ is a contraction
 with respect to the norm of the Sobolev space $H^t_\omega(M)$ and the identity
 $$
 (I - \mathcal B_u) \mathcal F_\xi (H, u) =0
 $$
indeed implies that $\mathcal F_\xi (H, u) =0$.

\smallskip
Finally we address the surjectivity of the differential at $H=H_0$ (and as a consequence $u=u_0$) of the function $P [u(H)]$, which implies that the function is a local
submersion, hence its zero locus is a $C^1$ submanifold.  The linearization of the para-cohomological equation \eqref{eq:para-cohom} reads
$$
 X_\xi v  +  \Big(  \sum_i d_i [\chi_i]  \Big)     = - D_H\mathcal{F}_\xi(H_0, u_0) (h)    \,.
$$
since $M[u_0]=Id$, $S[u_0]=0$,  and the terms $\mathcal R_{PL}$, $\mathcal R_{CM}$ are quadratic in $u-u_0$.  By Theorem~\ref{thm:CE}  and Corollary \ref{cor:CE},
 the above equation has a solution for the appropriate choice of the coefficients $(d_i)$ to ensure the vanishing of the obstructions. Such coefficients give the value of the
 differential of the map $P[u(H)]$ at the tangent vector $h$ representing the variation of the Hamiltonian $H$ at $H_0$. It is clearly possible to find a variation $h$ such that
 the values of the coefficients $(d_i)$ is any given vector of coefficients,  as long as the obstructions are chosen to be linearly independent functionals. This argument implies that the
 map $P[u(H)]$ is local submersion, as a function of $H$, with finite rank, hence its zero set is a local $C^1$ submanifold of finite codimension (by the implicit function theorem
 in Banach spaces).

\end{proof}

We conclude posing a couple of natural open questions:

\begin{question}  For  translation surfaces of higher genus and non-integrable rational billiards, are there arbitrarily small smooth perturbations such that the perturbed Hamiltonian does not have any 
invariant surface in the homology class of the zero section, or does not have any invariant surface which is a (Lipschitz) continuous graph over the zero section ?
\end{question}

\begin{question} For translation surfaces of higher genus and non-integrable rational billiards, are there any non-trivial smooth perturbations such that the perturbed Hamiltonian has a positive measure set
of the phase space foliated by invariant surfaces homologous to the zero section~? or at least an infinite number of distinct invariant surfaces homologous to the zero section~?
\end{question} 

In the completely integrable case,  by the KAM theory, the answer to the first question is negative and the answer to the second question is positive, for all sufficiently small smooth perturbation.

\end{document}